\documentclass{amsart}
\usepackage[utf8]{inputenc}
\usepackage{amsmath}
\usepackage{mathtools}
\usepackage{amsfonts}
\usepackage{amssymb}
\usepackage{amsthm} 
\usepackage{csquotes}
\usepackage[english]{babel}
\usepackage{lmodern}
\usepackage[shortlabels]{enumitem}
\usepackage{xcolor}
\usepackage[bookmarks=true]{hyperref}
\usepackage{graphicx}
\usepackage[T1]{fontenc}

\title{Canonical fibrations of contact metric $(\kappa,\mu)$-spaces}
\author{E. Loiudice and A. Lotta}
\date{}

\newtheorem{theorem}{Theorem}
\newtheorem{proposition}{Proposition}
\newtheorem{remark}{Remark}

\newtheorem{corollary}{Corollary}
\newtheorem{definition}{Definition}

\DeclareMathOperator{\id}{id}

\newcommand{\di}{\mbox{d}}
\newcommand{\kk}{\kappa}

\begin{document}

\begin{abstract} We present a classification of the complete, simply connected,  contact metric $(\kappa,\mu)$-spaces as homogeneous
contact metric manifolds, by studying the base space of their canonical fibration. 
According to the value of the Boeckx invariant, it turns out that the base is a complexification or a 
para-complexification of a sphere or of a hyperbolic space. 
In particular, we obtain a new homogeneous representation of the contact metric $(\kappa,\mu)$-spaces with Boeckx invariant less than $-1$.
\end{abstract}

\maketitle

\noindent
 {\small

}

\section{Introduction}




The study of the curvature tensor of associated metrics to a contact form is a central theme in contact metric geometry. 
Actually some important classes of contact metric manifolds can be defined using it.
%
We recall for example that Sasakian manifolds, the odd-dimensional analogues of K\"ahler manifolds, can be characterized by:
\begin{equation*}
 R(X,Y)\xi=\eta(Y)X-\eta(X)Y
\end{equation*}
where $X,Y$ are any vector fields and $\xi$ denotes the characteristic vector field of the contact metric manifold. 
A meaningful generalization of this curvature condition is 
\begin{equation*}
 R(X,Y)\xi=\kk\big(\eta(Y)X-\eta(X)Y\big)+\mu(\eta(Y)hX-\eta(X)hY\big)
\end{equation*}
where $\kk,\mu$ are real numbers and $2h$ is the Lie derivative of the structure tensor $\varphi$ in the direction of the characteristic vector field $\xi$.

The contact metric manifolds with this property were introduced by Blair, Koufogiorgos and Papantoniou 
in \cite{Blair K P}, and are called  {\em contact metric $(\kk,\mu)$--spaces} in the literature.
These spaces have
many interesting geometric properties; first of all, they are stable under $\mathcal{D}$-homothetic deformations and moreover in the non-Sasakian case, i.e., 
when $\kk\neq1$, the curvature tensor of the associated metric is completely determined. 
Looking at contact metric manifolds as strongly pseudo-convex (almost) $CR$ manifolds, in \cite{DileoLotta} Dileo and the second author  showed that
the $(\kk,\mu)$ condition is equivalent to the local $CR$-symmetry with respect to the Webster
metric, according to the general notion introduced by Kaup and Zaitsev in \cite{Kaup Zaitsev}.
In this context, another characterization was given by Boeckx and Cho in terms of the parallelism of the Tanaka--Webster curvature \cite{Boeckx-Cho}. 

\smallskip
Boeckx gave a crucial contribution to the problem of classifying these manifolds; 
after showing that every non-Sasakian contact $(\kk,\mu)$-space is locally homogeneous and strongly locally $\varphi$--symmetric \cite{Boeckx locally phi-symmetric},  
in \cite{Boeckx} he defined a scalar invariant $I_M$ which completely determines a contact $(\kk,\mu)$--space $M$ locally up 
to equivalence and up to $\mathcal{D}$-homotetic deformations of its contact metric structure. 

 \smallskip
A standard example is the tangent sphere bundle $T_1M$ of a Riemannian manifold $M$ with constant sectional curvature $c\not =1$.
Being an hypersurface of $TM$, which is equipped with a natural almost K\"ahler structure $(J,G)$, where $G$ is the Sasaki metric, $T_1M$ inherits a standard
contact metric structure (for more details, see for instance
\cite{Blair2010Book}). In particular, the Webster metric $g$ of $T_1M$ 
is a scalar multiple of $G$. The corresponding Boeckx invariant is given by:
$$
I_{T_1M}=\frac{1+c}{|1-c|}.
$$
Hence, as $c$ varies in $\mathbb{R} \smallsetminus \{1\}$, $I_{T_1M}$ assumes all the real values strictly greater than $-1$.

The case $I\leqslant -1$ seems to lead to models of different nature.
Namely, Boeckx found examples of contact metric $(\kk,\mu)$-spaces,  for every value of the invariant $I\leqslant -1$, namely
a two parameter family of (abstractly constructed) Lie groups with a left-invariant contact metric structure.
However, he gave no geometric description of these examples; in particular, to our knowledge, nothing can be found
in the literature regarding the topological structure of these manifolds.

\smallskip

One of the first aims of this paper is to fill this gap, showing that the simply connected, complete contact metric 
$(\kk,\mu)$--spaces with $I< -1$, with dimension $2n+1$ are exhausted by a one parameter family of invariant 
contact metric structures on the homogeneous space 
$$SO(n,2)/SO(n).$$
Actually, we provide a unified treatment of all the models with $I_M\not =\pm 1$.
Our classification is accomplished intrisically, by studying the canonical fibration of non Sasakian contact metric $(\kk,\mu)$-spaces with Boeckx invariant $I_M\neq \pm 1$ and endowing the base spaces of a canonical connection. 
Here we refer to the fibration $M\to M/\xi$
over the leaf space of the foliation determined by the Reeb vector field; as such, it depends only on the contact form of $M$.
First, in Theorem~\ref{characterization of (k,mu)}, non Sasakian contact metric $(\kk,\mu)$-spaces with Boeckx invariant $\neq \pm 1$ are characterized by admitting a transitive Lie group of automorphisms whose Lie algebra $\mathfrak{g}$ has a (canonical) symmetric decomposition. 
This decomposition yields a reductive decomposition for the base space $B$ of the canonical fibration and the associated 
canonical connection makes $B$ an affine symmetric space (Corollary~\ref{corb}).

Next we show that $B$ admits a uniquely determined
standard invariant complex or para-complex structure, by which it is a complexification or a para--complexification of the sphere $S^n$ 
or of the hyperbolic space $\mathbb{H}^n$, according to the value of the Boeckx invariant of the $(\kk,\mu)$--space. After identifying the 
possible base spaces $B$, in the final section we construct explicitly our models as homogeneous contact metric manifolds fiberings onto them.
In conclusion, we obtain the following classification list:

\begin{center}
\begin{tabular}{| c | c| c|}
\hline 
\multicolumn{3}{|c|}{Simply connected complete contact metric $(\kk,\mu)$-spaces with $I_M\not =\pm 1$} \\
\hline
Boeckx invariant & Model space       &    Base space                              \\ \hline
                 &                    &\\
 $I_M>1$         &   $SO(n+2)/SO(n)$ &    $SO(n+2)/(SO(n)\times SO(2))$   \\ 
                 &                    &   \\  \hline
                 &                    &\\
 $-1<I_M<1$      & $SO(n+1,1)/SO(n)$ &   $SO(n+1,1)/(SO(n)\times SO(1,1))$    \\ 
                 &                    &\\ \hline
                 &                    &\\  
 $I_M<-1$        & $SO(n,2)/SO(n)$    & $SO(n,2)/(SO(n)\times SO(2))$   \\ 
                 & &\\ \hline     
\end{tabular}
\end{center}

\medskip
This table also provides a new geometric interpretation of the Boeckx invariant.


\section{Preliminaries}
 
 Let $M$ be a odd-dimensional smooth manifold. An \emph{almost contact structure} on $M$ is a triple consisting of a $(1,1)$ tensor field $\varphi$, a vector field $\xi$ and a $1$-form $\eta$ satisfying:
 \begin{equation*}
 \varphi ^2=-\id+\eta \otimes \xi, \quad \eta(\xi)=1.
 \end{equation*}
 An almost contact manifold always admits a \emph{compatible metric}, namely a Riemannian metric $g$ such that
 \begin{equation*}
  g(\varphi X,\varphi Y)=g(X,Y)-\eta(X)\eta(Y),
 \end{equation*}
 for every vector fields $X,Y$ on $M$. If such a metric $g$ satisfies also
 \begin{equation*}
   \di\eta(X,Y)=g(X, \varphi Y),
 \end{equation*}
 then $(\varphi,\xi,\eta,g)$ is called a \emph{contact metric structure} on $M$. In this case $\eta$ is a contact
 form; we shall denote by $D$ the corresponding contact
 distribution $D=\mathrm{ker}(\eta)$ and by $\mathcal{D}$ the module of smooth sections of $D$.
 
 A contact metric manifold $M$ is said to be a \emph{$K$-contact manifold} if its characteristic vector field $\xi$ is Killing. 
 This condition is equivalent to the vanishing of the $(1,1)$ tensor field
 \begin{equation*}
  h:=\frac{1}{2}\mathcal{L}_{\xi}\varphi,
 \end{equation*}
 $\mathcal{L}_{\xi}$ being Lie differentiation in the direction of $\xi$.
 
 If the curvature tensor $R$ of a contact metric manifold $M$ satisfies the following condition
 \begin{equation*}
  R(X,Y)\xi=\eta(Y)X-\eta(X)Y,
 \end{equation*}
 for every vector fields $X,Y$ on $M$, then $M$ is a \emph{Sasakian} manifold. In this case $\xi$ is a Killing vector field and hence $M$ is a $K$-contact manifold.
 
 \medskip
  A \emph{contact metric $(\kk,\mu)$-space} is a contact metric manifold $(M,\varphi, \xi,\eta,g)$ such that
  \begin{equation*}
   R(X,Y)\xi=\kk(\eta(Y)X-\eta(X)Y)+\mu(\eta(Y)hX-\eta(X)hY),
  \end{equation*}
 where $X,Y\in \mathfrak{X}(M)$ are arbitrary vector fields and $\kk$, $\mu$ are real numbers.
 The $(\kk,\mu)$ condition is invariant under $D_a$-homothetic deformations. 
 We recall that a $D_a$-homothetic deformation of a contact metric manifold $(M, \varphi, \xi,\eta,g)$ is given by the following changing of the structural tensors of $M$:
 \begin{equation}
  \bar{\eta}:=a\eta, \quad \bar{\xi}:=\frac{1}{a}\xi,  \quad \bar{g}=ag+a(a-1)\eta \otimes \eta
 \end{equation}
 where $a$ is a positive constant.

 By direct computations one can check that a $D_a$-homothetic deformation transforms a contact metric $(\kk,\mu)$ space in a contact metric $(\bar{\kk},\bar{\mu})$ space where
 \begin{equation*}
  \bar{\kk}=\frac{\kk+a^2-1}{a^2}, \quad \bar{\mu}=\frac{\mu+2a-2}{a}.
 \end{equation*}
 In particular, a $D_a$-homothetic deformation of a contact metric manifold \; $(M,\varphi,\xi,\eta,g)$ satisfying $R(X,Y)\xi=0$ yields:
 \begin{equation*}
   \bar{R}(X,Y)\xi=\frac{a^2-1}{a^2}(\bar{\eta}(Y)X-\bar{\eta}(X)Y)+\frac{2a-2}{a}(\bar{\eta}(Y)\bar{h}X-\bar{\eta}(X)\bar{h}Y).
  \end{equation*}
 
  \smallskip
  
 In \cite{Blair K P} the authors proved the following Theorem.
 
 \begin{theorem}\label{th (k,mu) Blair}
  Let $(M,\varphi, \xi,\eta,g)$ be a contact metric $(\kk,\mu)$ manifold. Then $\kk\leqslant 1$. Moreover, if $\kk=1$ then $h=0$ and $(M,\varphi, \xi,\eta,g)$ is Sasakian. If $\kk<1$, the contact metric structure is not Sasakian and $M$ admits three mutually orthogonal integrable distributions $\mathcal{D}(0)$, $\mathcal{D}(\lambda)$ and $\mathcal{D}(-\lambda)$ corresponding to the eigenspaces of $h$, where $\lambda=\sqrt{1-\kk}$.
 \end{theorem}
 
 The explicit expression of the Riemannian curvature tensor of a non-Sasakian contact
 metric $(\kk,\mu)$-manifold is known (see \cite[Theorem~5]{Boeckx locally phi-symmetric}):
 
 \begin{theorem}\label{(k,mu) curvature}
  Let $M$ be a contact metric $(\kk,\mu)$-space. If  $\kk\neq1$ then:
 \begin{equation*}
  \begin{aligned}
   g(R(X, Y)Z, W) =& \Big(1-\frac{\mu}{2} \Big)(g(Y, Z)g(X, W)- g(X, Z)g(Y, W))\\
                             &+ g(Y, Z)g(hX, W)-g(X, Z)g(hY, W)\\
                             & -g(Y, W)g(hX, Z) + g(X, W)g(hY, Z)\\
                             & \frac{1-\mu / 2}{1-\kk}(g(hY, Z)g(hX, W)- g(hX, Z)g(hY, W))\\
                             & -\frac{\mu}{2} (g(\varphi Y, Z)g(\varphi X, W)- g(\varphi X, Z)g(\varphi Y, W))\\
                             & \frac{\kk-\mu / 2}{1-\kk} (g(\varphi hY, Z)g(\varphi hX, W) - g(\varphi hY, W)g(\varphi hX, Z))\\
                             & +\mu g(\varphi X, Y)g(\varphi Z, W)\\
                             &+\eta(X)\eta(W)\Big( \Big( \kk- 1 + \frac{\mu}{2}\Big)g(Y, Z) + (\mu-1)g(hY, Z)\Big)\\
                             &-\eta(X)\eta(Z)\Big( \Big( \kk- 1 + \frac{\mu}{2}\Big)g(Y, W) + (\mu-1)g(hY, W)\Big)\\
                             &+\eta(Y)\eta(Z)\Big( \Big( \kk- 1 + \frac{\mu}{2}\Big)g(X, W) + (\mu-1)g(hX, W)\Big)\\
                             &-\eta(Y)\eta(W)\Big( \Big( \kk- 1 + \frac{\mu}{2}\Big)g(X, Z) + (\mu-1)g(hX, Z)\Big),\\
  \end{aligned}
 \end{equation*}
 \end{theorem}

 \bigskip
 The class of non-Sasakian contact metric $(\kk,\mu)$-spaces coincides with the class of contact metric manifolds
 with non vanishing $\eta$-parallel tensor $h$, according to \cite[Lemma~3.8]{Blair K P} and the following result of  Boeckx and Cho \cite{Bo-Cho}:
 \begin{theorem}\label{Bo-Ch}
  Let $(M,\varphi, \xi,\eta,g)$ be a contact metric manifold which is not $K$-contact. If $g((\nabla_Xh)Y,Z)=0$ for every vector fields $X,Y,Z$ orthogonal to $\xi$, then $M$ is a contact metric $(\kk,\mu)$-space. 
 \end{theorem}

 Finally, we recall also the following characterization in the context of $CR$ geometry (we refer to \cite[Section~6.4]{Blair2010Book} and \cite{LibroDragomirTomassini} for a general reference on this topic):

 \begin{theorem}\cite[Theorem~3.2]{DileoLotta}\label{DileoL}
  Let $(M, HM, J, \eta)$ be a pseudo-Hermitian manifold. Assume that the Webster metric $g_{\eta}$ is not Sasakian. The following conditions are equivalent:
  \begin{enumerate}[(1)]
   \item The Webster metric $g_{\eta}$ is locally $CR$-symmetric.
   \item The underlying contact metric structure satisfies the $(\kk, \mu)$ condition.
  \end{enumerate}
 \end{theorem}

 \bigskip
 Non-Sasakian contact metric $(\kk,\mu)$-spaces have been completely classified by Boeckx in \cite{Boeckx}. In this case  $\kk<1$ and the real number 
 $$
 I_M:=\frac{1-\frac{\mu}{2}}{\sqrt{1-\kk}},
 $$
 is an invariant for the $(\kk,\mu)$ structure, that we call \emph{Boeckx invariant}. Indeed we have that:
 \begin{theorem}[\cite{Boeckx}]
 Let $(M_i, \varphi_i, \xi_i, \eta_i, g_i )$, $i=1, 2$, be two non-Sasakian $(\kk_i, \mu_i)$-spaces of the same dimension. Then $I_{M_1}=I_{M_2}$ if and only if, up to a $D$-homothetic
 transformation, the two spaces are locally isometric as contact metric spaces. In particular, if both spaces are simply connected and complete, they are globally isometric up to a $D$-homothetic transformation.
 \end{theorem}

\medskip
Next we recall the notions of straight and twisted complexifications of a Lie Triple System (LTS). For more details we refer the
reader to \cite{Bertram} and \cite{Bertram2}. Given a Lie triple system $(\mathfrak{m},[\,,\,,\,])$ we shall write as usual
$$R(X,Y)Z:=-[X,Y,Z].$$
 We shall also write $(\mathfrak{m},R)$ instead of $(\mathfrak{m},[\,,\,,\,])$.
 An \emph{invariant complex structure} on $\mathfrak{m}$ is a complex structure $J:\mathfrak{m}\to\mathfrak{m}$ such that for every $X,Y,Z\in \mathfrak{m}:$
 $$
 [X,Y,JZ]=J[X,Y,Z].
 $$

An \emph{invariant para-complex structure} $I$ on $\mathfrak{m}$ is a para-complex structure on $\mathfrak{m}$ (i.e., an endomorphism of $\mathfrak{m}$ such that $I^2=\id_{\mathfrak{m}}$ and the $\pm 1$ eigenspaces of $I$ have the same dimension) satisfying: 
$$
[X,Y,IZ]=I[X,Y,Z],
$$
for every $X,Y,Z\in \mathfrak{m}$.

For every LTS $\mathfrak{m}$ endowed with an invariant (para-)complex structure, the corresponding simply connected symmetric space $G/H$ is canonically endowed with a $G$-invariant almost (para-)complex structure and viceversa (see Proposition III.1.4 of \cite{Bertram}). 
%

 An invariant (para-)complex structure $J$ on a LTS $(\mathfrak{m},[\,,\,,\,])$ is called \emph{straight} or \emph{twisted} respectively if:
 \begin{equation*}
  [JX,Y,Z]=[X,JY,Z]
 \end{equation*}
 or
\begin{equation*}
  [JX,Y,Z]=-[X,JY,Z].
 \end{equation*}

Accordingly, a \emph{straight} or respectively \emph{twisted} (para-)complex symmetric space is an affine symmetric space $M=G/H$ endowed with an invariant almost (para-) complex structure $\mathcal{J}$ such that 
$$R(\mathcal{J}X,Y)Z=R(X,\mathcal{J}Y)Z,$$ 
or respectively
$$R(\mathcal{J}X,Y)Z=-R(X,\mathcal{J}Y)Z,$$
where $R$ is the curvature of $M$. 

\medskip
 A \emph{(para-)complexification} of a LTS $\mathfrak{m}$ is a LTS $(\mathfrak{q},[\,,\,,\,])$ together with an invariant (para-)complex structure $J$ and an automorphism $\tau$ such that $\tau J+J\tau=0$, $\tau^2=\id_{\mathfrak{q}}$ and the LTS $\mathfrak{q}^{\tau}$ given by the 
 space of $\tau$-fixed points of $\mathfrak{q}$ is isomorphic to $\mathfrak{m}$. The (para-)complexification $(\mathfrak{q},[\,,\,,\,],J,\tau)$ of $\mathfrak{m}$ is called \emph{straight} or \emph{twisted} respectively if $J$ is a straight or twisted.

%

\bigskip

We recall that every LTS $(\mathfrak{m},R)$ has a unique straight complexification given by the $\mathbb{C}$--trilinear extension
$R_{\mathbb{C}}:\mathfrak{m}_{\mathbb{C}}\times \mathfrak{m}_{\mathbb{C}}\times \mathfrak{m}_{\mathbb{C}} \rightarrow 
\mathfrak{m}_{\mathbb{C}}
$ of $R$ \cite[Proposition~2.1.4]{Bertram2}.
The existence of a twisted complexification or para-complexification of  $\mathfrak{m}$ is instead related to the existence of a particular $(1,3)$-tensor, the \emph{Jordan extension} of $R$.

Let $M=G/H$ be a symmetric space endowed with an invariant almost (para-) complex structure $\mathcal{J}$.
 The \emph{structure tensor} of $\mathcal{J}$ is the $(1,3)$-tensor
 \begin{equation*}
  T(X,Y)Z=-\frac{1}{2}\big( R(X,Y)Z-\mathcal{J}R(X,\mathcal{J}^{-1}Y)Z\big).
 \end{equation*}
 This tensor satisfies the following two properties:
 \begin{equation}\tag{JT1}\label{JT1}
  T(X,Y)Z=T(Z,Y)X
 \end{equation}
\begin{equation}\tag{JT2}\label{JT2}
 \begin{aligned}
  T(U,V)T(X,Y,Z)=&T(T(U,V)X,Y,Z)-T(X,T(U,V)Y,Z)\\
                 &+T(X,Y,T(U,V)Z),
 \end{aligned}
\end{equation}
Now, a \emph{Jordan triple system} is a pair $(V,T)$, where $V$ is a vector space and $T:V\times V \times V\rightarrow V$ is a trilinear map satisfying \eqref{JT1}, \eqref{JT2},  called a \emph{Jordan triple product} on $V$.

Observe that if $T$ is a JT product on $V$, then
$$
[x,y,z]:=T(x,y)z-T(y,x)z
$$
is a LT product on $V$. 

 Let $T$ be a JT product on a LTS $(\mathfrak{m}, R)$. We set
 $$R_T(x,y):=-T(x,y)+T(y,x).$$
  $T$ is said to be a \emph{Jordan extension} of $R$ if $R=R_T$.

\begin{theorem}\cite[Theorem~III.4.4]{Bertram}\label{Thjor}
 Let $(\mathfrak{m},R)$ be a LTS. The following object are in one-to-one correspondence:
 \begin{enumerate}[(1)]
  \item twisted complexification of $R$,
  \item twisted para-complexisication of $R$,
  \item Jordan extension of $R$.
 \end{enumerate}
\end{theorem}

\bigskip

In the next section we shall be concerned with the following basic examples, studying their interplay
with the classification of contact metric $(\kk,\mu)$-manifolds.
Consider the Lie triple systems $(\mathbb{R}^n,R)$ and $(\mathbb{R}^n,-R)$, associated respectively to the sphere 
$S^n$ and the hyperbolic space $\mathbb{H}^n$, where $R$ is
$$
R(x,y)z:=<y,z>x-<x,z>y.
$$

On $(\mathbb{R}^n,R)$ one can consider the following JT product:
$$
T(x,y)z=<x,z>y-<x,y>z-<y,z>x.
$$

Then, according to  \cite[Proposition~IV.1.5]{Bertram},
the corresponding twisted complexification and para-complexification of $S^n$, 
are the the symmetric spaces
$$SO(n+2)/(SO(n)\times SO(2))$$
and 
$$SO(n+1,1)/(SO(n)\times SO(1,1)).$$

\medskip
In the case of $\mathbb{H}^n$, one can consider $-T$; 
the corresponding twisted complexification is  (see \cite[ p.~91]{Bertram}):
$$SO(n,2)/(SO(n)\times SO(2)).$$

\section{A characterization of contact metric $(\kappa,\mu)$-spaces}\label{km}

%

%

\bigskip
Let $(M,\varphi,\xi,\eta,g)$ be a connected homogeneous contact metric manifold. Consider a Lie group $G$ acting transitively on $M$ as a group of automorphisms of the contact metric structure, and denote by $H$ the isotropy subgroup of $G$ at $x_o\in M$. The natural map $j:G/H\rightarrow M$ given by $j(aH)=ax_o$ is a diffeomorphism. Thus $G/H$ is a homogeneous Riemannian space and in particular it is a reductive homogeneous space (cf. e.g. \cite{TricerriVan}). Fix a reductive decomposition of the Lie algebra $\mathfrak{g}$ of $G$:
\begin{equation}\label{h+m}
 \mathfrak{g}=\mathfrak{h}\oplus \mathfrak{m},
\end{equation}
where $\mathfrak{h}=Lie(H)$.
The identity component $G^o$ of $G$ acts again transitively on $M$, and the isotropy subgroup of $G^o$ at $x_o$ is $H\cap G^o $. Let
$$
\pi:G^o\rightarrow G^o / H\cap G^o \simeq M
$$ 
be the natural fibration of $G^o$ onto the homogeneous space $G^o / H\cap G^o$. 
Being $Lie(H)=Lie(H\cap G^o)$, \eqref{h+m} is also a reductive decomposition for $G^o / H\cap G^o$.
Then $\mathfrak{m}$ decomposes into the direct sum of two $H\cap G^o$-invariant subspaces
$$
\mathfrak{m}=\mathbb{R}J\oplus \mathfrak{b},
$$
where $J$ is the vector of $\mathfrak{m}$ corresponding to $\xi_{\underline{o}}$ and $\mathfrak{b}$ corresponds to the determination of the contact distribution $D=\ker(\eta)$ at $\underline{o}:=\pi(e)\cong x_o$, being $e$ the neutral element of $G$.

Now, homogeneity ensures that the contact form $\eta$ is regular (see \cite[Section~II]{BoothWang}); hence we have a canonical fibration of $M$,
given by (see also \cite[p. 225]{Musso}):
\begin{equation*}
  G^o/H\cap G^o \rightarrow G^o/S^o(H\cap G^o),
\end{equation*}
where $S^o$ is the identity component of the closed Lie subgroup 
\begin{equation*}
 S:=\{ h\in G^o \;|\; Ad(h)^*\tilde{\eta}=\tilde{\eta} \}
\end{equation*}
 of $G^o$. Here  $\tilde{\eta}$ denotes the one form on $G^o$ pull back of $\eta$ via $\pi$. We have that $H\cap G^o\subset S$ (\cite[Lemma~II.4]{BoothWang}).
 
Moreover, the Lie algebra $\bar{\mathfrak{h}}$ of $\bar{H}:=S^o(H\cap G^o)$ is given by:
$$
\bar{\mathfrak{h}}=\mathfrak{h}\oplus \mathbb{R}J,
$$
and we have the following decomposition of $\mathfrak{g}$:

\begin{equation}\label{h+b}
\mathfrak{g}=\bar{\mathfrak{h}}\oplus \mathfrak{b}.
\end{equation}

\bigskip


Our first aim is to characterize the non-Sasakian contact metric $(\kappa,\mu)$-spaces as homogeneous contact metric manifolds for which decomposition \eqref{h+b} is \emph{symmetric}, i.e.,  
$$[\bar{\mathfrak{h}},\bar{\mathfrak{h}}] \subset \bar{\mathfrak{h}},\,
[\bar{\mathfrak{h}},\mathfrak{b}] \subset \mathfrak{b},\,[\mathfrak{b},\mathfrak{b}]\subset \bar{\mathfrak{h}}.$$

Using this, in Corollary~\ref{corb}, we shall be able to endow $B$ of $G^o$-invariant affine connections making it an affine symmetric space.

\bigskip

\begin{theorem}\label{characterization of (k,mu)}
Let $(M,\varphi, \xi,\eta,g)$ be a  simply connected, complete, contact metric manifold. Assume $M$ is not  $K$-contact.
Then the following conditions are equivalent:
\begin{enumerate} [(a)]
 \item $M$ is a contact metric $(\kappa,\mu)$-space.
 \item $M$ admits a transitive, effective Lie group of automorphisms $G$
 whose Lie algebra $\mathfrak{g}$ is a symmetric Lie algebra with symmetric decomposition \eqref{h+b}.\label{homogg}
\end{enumerate}

\end{theorem}

\begin{proof}
(a) $\Rightarrow$ (b).
According to \cite{Boeckx locally phi-symmetric}, $(M,\varphi,\xi,\eta,g)$ is a homogeneous contact metric manifold.
Let $G=Aut(M)$ be the Lie group of all the automorphisms of the contact metric structure of $M$, and $H$ be the isotropy subgroup of $G$ at $x_o\in M$. 

We fix a reductive decomposition of $\mathfrak{g}$
\begin{equation}\label{reductive dec}
\mathfrak{g}=\mathfrak{h}\oplus \mathfrak{m},
\end{equation}
where $\mathfrak{g}$ and $\mathfrak{h}$ are respectively the Lie algebras of $G$ and $H$. Keeping the notation above we consider also the decompositions
$$
\mathfrak{g}=\mathfrak{h}\oplus\mathbb{R}J\oplus \mathfrak{b}=\bar{\mathfrak{h}}\oplus \mathfrak{b}.
$$ 

By Theorem \ref{DileoL}, for every $x\in M$ there exists a local $CR$-symmetry at $x$. Being $M$ simply connected and complete, the local $CR$-symmetries are actually globally defined.
Let $\sigma$ be the $CR$-symmetry at $\underline{o}=eH$.  We recall that $\sigma$ is an isometric $CR$ diffeomorphism of $M$, whose
differential at $\underline{o}$ is $-\mathrm{Id}$ on ${D}_{\underline{o}}$.
In particular, it is a automorphism of the contact metric structure and and  affine automorphism of the canonical $G$-invariant affine connection $\tilde{\nabla}$ associated to \eqref{reductive dec}. 
Hence, denoting by $\tilde{T}$ the torsion of $\tilde{\nabla}$, we have that, for every $X,Y,Z\in \mathfrak{b}\subset \mathfrak{m}$:
\begin{equation*}
\begin{aligned}
 g_{\underline{o}}(\tilde{T}(X,Y),Z)&= g_{\underline{o}}(\sigma_{\star}\tilde{T}(X,Y),\sigma_{\star}Z)\\
                             & =g_{\underline{o}}(\tilde{T}(\sigma_{\star}X,\sigma_{\star}Y),\sigma_{\star}Z)\\
                             &= - g_{\underline{o}}(\tilde{T}(X,Y),Z),
\end{aligned}
\end{equation*}
which yields that $[X,Y]_{\mathfrak{m}}=-\tilde{T}_{\underline{o}}(X,Y) \in \mathbb{R}$, and hence $[\mathfrak{b},\mathfrak{b}]\subset \bar{\mathfrak{h}}$.
\medskip

The curvature tensor $\tilde{R}$ of $\tilde{\nabla}$ and the Reeb vector field $\xi$ are also preserved by $\sigma$. Hence for every $X,Y,Z\in\mathfrak{b}$:
\begin{equation*}
\begin{aligned}
 g_{\underline{o}}(\tilde{R}(J,X)Y,Z)&= g_{\underline{o}}(\sigma_{\star}\tilde{R}(J,X)Y,\sigma_{\star}Z)\\
                             & =g_{\underline{o}}(\tilde{R}(\sigma_{\star}J,\sigma_{\star}X)\sigma_{\star}Y,\sigma_{\star}Z)\\
                             &= - g_{\underline{o}}(\tilde{R}(J,X)Y,Z),
\end{aligned}
\end{equation*}
moreover, since $\tilde{\nabla}\mathcal{D} \subset \mathcal{D}$ we have that $\tilde{R}(J,X)Y\in \mathcal{D}_{\underline{o}};$ thus
$$
[[J,X]_{\mathfrak{h}},Y]=0,
$$
for every $X,Y\in \mathfrak{b}$. Being $G$ effective on $M$, the adjoint representation
$ad:\mathfrak{h}\rightarrow End({\mathfrak{m}})$
is injective; therefore, using also $[\mathfrak{h},J]=0$, we conclude that $[J,X]_{\mathfrak{h}}=0$.

Finally we prove that $[J,X]\in\mathfrak{b}$; indeed we have:
\begin{equation*}
\begin{aligned}
 g_{\underline{o}}(\tilde{T}(J,X),J)&= g_{\underline{o}}(\sigma_{\star}\tilde{T}(J,X),\sigma_{\star}J)\\
                             & =g_{\underline{o}}(\tilde{T}(\sigma_{\star}J,\sigma_{\star}X),\sigma_{\star}J)\\
                             &= - g_{\underline{o}}(\tilde{T}(J,X),J),
\end{aligned}
\end{equation*}
This completes the proof of \ref{homogg}.

\bigskip

(b) $\Rightarrow$ (a). 
Let $\mathfrak{g}=\mathfrak{h}\oplus \mathfrak{m}$ be a reductive decomposition for the homogeneous contact metric space $M=G/H$,  where $H$ is the isotropy subgroup of $G$ at a point $x_o\in M$.

Let $\nabla$ and $\tilde{\nabla}$ respectively the Levi-Civita connection of $g$ and the canonical affine connection on $M$ associated to the fixed reductive decomposition. If we set $A=\nabla -\tilde{\nabla}$, then
\begin{equation*}
 (\nabla_Xh)Y=-(\tilde{\nabla}_Xh)Y+A(X,hY)-hA(X,Y).
\end{equation*}
Now, since the tensor $h=\frac{1}{2}\mathcal{L}_{\xi}\varphi$ is invariant under automorphisms of the
 contact metric structure,  it is parallel with respect to the canonical connection $\tilde{\nabla}$ (\cite[p.~193]{KN}) and hence:
\begin{equation}\label{nablah}
 (\nabla_Xh)Y=A(X,hY)-hA(X,Y).
\end{equation}
Being $\tilde{\nabla}$ a metric connection, for $X,Y,Z\in \mathfrak{X}(M)$ we have that
\begin{equation}\label{g(A)}
 g(A(X,Y),Z)+g(Y,A(X,Z))=0.
\end{equation}
Then for every $X,Y,Z\in \mathfrak{X}(M)$:
\begin{equation}\label{gA}
 2g(A(X,Y),Z)=-g(\tilde{T}(X,Y),Z)+g(\tilde{T}(Y,Z),X)-g(\tilde{T}(Z,X),Y).
\end{equation}
Now observe that for every $X,Y\in \mathfrak{b}$:
$$
\tilde{T}_{\underline{o}}(X,Y)=-[X,Y]_{\mathfrak{m}},
$$
and  
$$
[X,Y]\in \mathfrak{h}\oplus \mathbb{R}J,
$$
being $\mathfrak{g}=\bar{\mathfrak{h}}\oplus \mathfrak{b}$ a symmetric decomposition by assumption. Thus $\tilde{T}_{\underline{o}}(X,Y)\in \mathbb{R}J$. Hence for every $X,Y,Z\in \mathcal{D}$:
$$
g(\tilde{T}(X,Y),Z)=0,
$$
and then, by \eqref{gA}
$$
g(A(X,Y),Z)=0.
$$ 
Thus, using \eqref{nablah}, we obtain that
$$
g((\nabla_Xh)Y,Z)=0,
$$
for every $X,Y,Z\in \mathcal{D}$. This implies that $M$ is a contact metric $(\kappa,\mu)$-space according to 
Theorem \ref{Bo-Ch}.
\end{proof}

\begin{corollary}\label{corb}
 Let $M=G/H$ be a simply connected, complete, non-Sasakian contact metric $(\kappa,\mu)$-manifold. 
 Then the base space $B=G^o/\bar{H}$ of the canonical fibration of $M$ is an an affine symmetric space. 
%
%
\end{corollary}
\begin{proof}

It suffices to prove that $B=G^o/\bar{H}$ is a homogeneous reductive space with respect to decomposition \eqref{h+b};
indeed, the associated canonical $G^o$-invariant connection makes $B$ a locally symmetric affine manifold. Observe that $B$ is simply connected since the fibers of the canonical fibration are connected
(cf. \cite[Theorem II.4]{BoothWang}). Being the canonical invariant connection always complete (see Corollary~2.5 of \cite[Chapter~X]{KN}), $B$ is actually a symmetric space.

To prove our claim,  we already recalled that
$H\cap G^o \subset S$; thus $S^o\subset S^o(H\cap G^o) \subset S$ and $Lie(S^o)=\bar{\mathfrak{h}}$. 
Since $[\bar{\mathfrak{h}},\mathfrak{b}]\subset \mathfrak{b}$ 
and $S^o$ is connected, it follows that $Ad(S^o)\mathfrak{b}\subset \mathfrak{b}$ 
and hence, being also $Ad(H\cap G^o)(\mathfrak{b})\subset \mathfrak{b}$, we conclude
that $Ad(\bar{H})\mathfrak{b}\subset \mathfrak{b}$ as claimed.

%
 
\end{proof}
%
%

We remark that the affine symmetric structure on $B$ thus obtained a priori
depends on the initial choice of a reductive decomposition \eqref{h+m} of $\mathfrak{g}$.
In the next section, we shall see that actually different choices lead to the
same affine symmetric space, up to isomorphism (see  Corollary~\ref{unicita}).

\section{The base space of the canonical fibration}

The aim of this section is to give a complete classification of the symmetric base spaces $B$
of the canonical fibrations of simply connected, complete, non-Sasakian contact metric $(\kappa,\mu)$-manifolds with Boeckx invariant $I_M\neq\pm1$. We obtain that $B$ is a twisted complexfication
or para-complexification of the sphere $S^n$, or of the hyperbolic space $\mathbb{H}^n$ according to the following table: 
\begin{table}[!h]
\caption{}
\label{tab:tableB}
\scalebox{0.9}{%
\begin{tabular}{| c | c| c|} 
\hline
Boeckx invariant &    Base space                        &   Type      \\ \hline
                 &                                      &\\
 $I_M>1$         &    $SO(n+2)/(SO(n)\times SO(2))$     &   Complexification of $S^n$\\ 
                 &                                      &          \\  \hline
                 &                                      &\\
 $-1<I_M<1$      &   $SO(n+1,1)/(SO(n)\times SO(1,1))$  &  Para-complexification of $S^n$\\ 
                 &                                      &   \\ \hline
                 &                                      & \\  
 $I_M<-1$        &  $SO(n,2)/(SO(n)\times SO(2))$     &     Complexification of $\mathbb{H}^n$  \\ 
                 &                                     &  \\ \hline     
\end{tabular}}
\end{table}

%

\vspace{2mm}

\bigskip

Keeping the notations above, we identify the tangent space of $B$ at the base point with the linear subspace $\mathfrak{b}\cong\mathcal{D}_o$.
Moreover we denote by $\mathfrak{b}_+$ and $\mathfrak{b}_-$ the subspaces of $\mathfrak{b}$ corresponding respectively to the eigenspaces $\mathcal{D}_o(\lambda)$ and $\mathcal{D}_o(-\lambda)$ of $h_o:\mathfrak{b}\to\mathfrak{b}$.

\medskip
We start by computing the curvature of $B$.

\begin{proposition}\label{base space curvature}
Let $(M,\varphi, \xi,\eta,g)$ be a simply connected, complete, non-Sasakian contact metric $(\kappa,\mu)$-manifold and $B$ the base space of the canonical fibration of $M$. If $\bar{\nabla}$ is the canonical affine connection on $B$ associated to any reductive decomposition of type \eqref{h+b}, then the curvature tensor $\bar{R}$ of $\bar{\nabla}$ at the base point $o\in B$ is given by

 \begin{equation}\label{R symmetric base}
  \begin{aligned}
   \bar{R}_o(X,Y)Z  =&  \Big((1-\frac{\mu}{2})g(Y,Z)+g(hY,Z)\Big)X\\
                   & -\Big((1-\frac{\mu}{2})g(X,Z)+g(hX,Z)\Big)Y\\
                   & +\Big(\frac{1-\frac{\mu}{2}}{1-\kappa}g(hY,Z)+g(Y,Z)\Big)hX\\
                   & -\Big(\frac{1-\frac{\mu}{2}}{1-\kappa}g(hX,Z)+g(X,Z)\Big)hY\\
                   & +\Big((1-\frac{\mu}{2})g(\varphi Y,Z)+g(\varphi hY,Z)\Big)\varphi X\\
                   & -\Big((1-\frac{\mu}{2})g(\varphi X,Z)+g(\varphi hX,Z)\Big)\varphi Y\\
                   & +\Big(\frac{1-\frac{\mu}{2}}{1-\kappa}g(\varphi hY,Z)+g(\varphi Y,Z)\Big)\varphi hX\\
                   & -\Big(\frac{1-\frac{\mu}{2}}{1-\kappa}g(\varphi hX,Z)+g(\varphi X,Z)\Big)\varphi hY\\
                   & +(\mu-2)g(\varphi X,Y)\varphi Z-2g(\varphi X,Y)\varphi hZ.
  \end{aligned}
 \end{equation}
\end{proposition}

\begin{proof}

For every $X,Y,Z\in \mathfrak{b}$ we have (see \cite[Chapter~X]{KN}):
%
\begin{equation*}
 \bar{R}_o(X,Y)Z=-[[X,Y]_{J}+[X,Y]_{\mathfrak{h}},Z],
\end{equation*}
and hence 
\begin{equation}\label{R_R}
 \bar{R}_o(X,Y)Z=\tilde{R}(X,Y)Z-[[X,Y]_{J},Z],
\end{equation}
where $[X,Y]_{J}$ and $[X,Y]_{\mathfrak{h}}$ are the components of $[X,Y]\in \mathfrak{g}=\mathfrak{h}\oplus\mathbb{R}J\oplus\mathfrak{b}$ respectively in $\mathbb{R}J$ and $\mathfrak{h}$; being $\tilde{R}$ the curvature tensor of the canonical connection of the homogeneous reductive space $M$ with reductive decomposition $\mathfrak{g}=\mathfrak{h}\oplus \mathfrak{m}$.

Let $\nabla$ be the Levi-Civita connection of $g$ and $R$ the curvature tensor of $\nabla$. If we set $A:=\tilde{\nabla}-\nabla$, then a standard computation yields:
\begin{equation*}
\begin{aligned}
 \tilde{R}(X,Y)Z = & R(X,Y)Z-A(X,A(Y,Z))+A(Y,A(X,Z))\\
                   &+A(\tilde{T}(X,Y),Z)+(\tilde{\nabla}_XA)(Y,Z)-(\tilde{\nabla}_YA)(X,Z),
 \end{aligned}
\end{equation*}
for every $X,Y,Z\in\mathfrak{X}(M)$.
Moreover, being $A$ a $G$-invariant tensor, we have that $A$ is parallel with respect to the canonical connection $\tilde{\nabla}$ 
and hence
\begin{equation*}
 \tilde{R}(X,Y)Z =  R(X,Y)Z-A(X,A(Y,Z))+A(Y,A(X,Z))+A(\tilde{T}(X,Y),Z),
\end{equation*}
and equation \eqref{R_R} becomes:
\begin{equation*}
\begin{aligned}
 \bar{R}_o(X,Y)Z=& R(X,Y)Z-A(X,A(Y,Z))+A(Y,A(X,Z))\\
               & +A(\tilde{T}(X,Y),Z)-[[X,Y]_{J},Z].
\end{aligned}               
\end{equation*}
We already observed in the proof of Theorem~\ref{characterization of (k,mu)} that for every $X,Y,Z\in \mathcal{D}$:
\begin{equation*}
  g(A(X,Y),Z)=0, \quad g(\tilde{T}(X,Y),Z)=0;
\end{equation*}
 hence
\begin{equation}\label{Ag}
A(X,Y)=g(A(X,Y),\xi)\xi,
\end{equation}
\begin{equation}\label{Tg}
 \tilde{T}(X,Y)=g(\tilde{T}(X,Y),\xi)\xi=-g([X,Y],\xi)\xi=2g(X, \varphi Y)\xi.
\end{equation}
In \eqref{Tg} we are using the parallelism of the distributions $\mathcal{D}(\pm \lambda)$ with respect to $\tilde{\nabla}$, which is a consequence of the fact that $\tilde{\nabla}h=0$.

Moreover we have:
\begin{equation}\label{An}
A(X,\xi)=\tilde{\nabla}_X\xi-\nabla_X\xi=\varphi X+\varphi hX.
\end{equation}
Then, using \eqref{Ag}, \eqref{Tg}, \eqref{An}, specializing at the point $o$ we obtain:
\begin{equation}\label{bR}
 \begin{aligned}
  \bar{R}_o(X,Y)Z = & R(X,Y)Z-g(A(Y,Z),J)A(X,J)+g(A(X,Z),J)A(Y,J)\\
                  & +2g(X,\varphi Y)A(J,Z)+[\tilde{T}(X,Y),Z]\\
                = & R(X,Y)Z-g(A(Y,Z),J)(\varphi X+\varphi hX)\\
                  & +g(A(X,Z),J)(\varphi Y+\varphi hY)+2g(X,\varphi Y)A(J,Z)\\
                  & +2g(X,\varphi Y)[J,Z],\\
 \end{aligned}
\end{equation}
where $X,Y,Z\in \mathfrak{b}$. 
The $(1,1)$-tensor $A(X,\cdot)$ is a skew symmetric tensor, being $\tilde{\nabla}g=0$. In particular
\begin{equation*}
 g(A(X,Y),\xi)=-g(Y,A(X,\xi)),
\end{equation*}
so that, by \eqref{An}:
\begin{equation*}
 g(A(X,Y),\xi)=-g(Y,\varphi X+\varphi hX).
\end{equation*}
Thus, equation \eqref{bR} becomes:
\begin{equation*}
 \begin{aligned}
  \bar{R}_o(X,Y)Z = & R(X,Y)Z+g(Z,\varphi Y+\varphi hY)(\varphi X+\varphi hX)\\
                  & -g(\varphi X+\varphi hX,Z)(\varphi Y+\varphi hY)+2g(X,\varphi Y)A(J,Z)\\
                  & +2g(X,\varphi Y)[J,Z].  
 \end{aligned}
\end{equation*}
%
%
%
Now, using Theorem~\ref{characterization of (k,mu)}: 
\begin{equation*}
\tilde{T}_o(J,Z)=-[J,Z]_{\mathfrak{m}}=-[J,Z];
\end{equation*}
on the other hand
\begin{equation*}
\begin{aligned}
 \tilde{T}(\xi,W) =& \tilde{\nabla}_{\xi}W-\tilde{\nabla}_W \xi-[\xi,W]\\
                  =& \nabla_{\xi}W+A(\xi,W)-[\xi,W]\\
                  =& -\varphi W-\varphi hW+A(\xi,W),
 \end{aligned}
\end{equation*}
for every $W$ vector field on $M$. Thus:
\begin{equation*}
 \begin{aligned}
  \bar{R}_o(X,Y)Z = & R(X,Y)Z+g(Z,\varphi Y+\varphi hY)(\varphi X+\varphi hX)\\
                  & -g(\varphi X+\varphi hX,Z)(\varphi Y+\varphi hY)+2g(X,\varphi Y)A(J,Z)\\
                  & -2g(X,\varphi Y)(-\varphi Z-\varphi hZ+A(J,Z))\\  
                = &   R(X,Y)Z+g(Z,\varphi Y+\varphi hY)(\varphi X+\varphi hX)\\
                  & -g(\varphi X+\varphi hX,Z)(\varphi Y+\varphi hY)+2g(X,\varphi Y)(\varphi Z+\varphi hZ).\\   
 \end{aligned}
\end{equation*}
Finally, taking into account the explicit expression of the curvature tensor $R$ of 
$M$ (see Theorem~\ref{(k,mu) curvature}), we obtain \eqref{R symmetric base}.
\end{proof}

\begin{corollary}\label{unicita}
 The affine base spaces $(B,\bar{\nabla})$ of a simply connected, complete, non-Sasakian, contact metric $(\kappa,\mu)$-manifold are 
all mutually equivalent affine symmetric spaces. 
\end{corollary}

 \bigskip
 For a non-Sasakian contact metric $(\kappa,\mu)$-space the restriction of the $(1,1)$ tensor $\varphi$ to the horizontal distribution 
 does not induce a complex structure on the base space, as occurs in the homogeneous Sasakian case, because $h\not=0$.
 However, we shall see in the following that $B$ admits a \emph{standard} complex or para-complex structure, according to the following definition and Theorem~\ref{th_standard}.

\begin{definition}\label{def standard} \em
 Let $(M,\varphi,\xi,\eta,g)$ be a contact metric $(\kappa,\mu)$-manifold and $(B,\bar{\nabla})$ the base space of the canonical fibration  of $M$.
 
A $G^o$-invariant almost complex structure $\mathcal{J}$ on $B$ will be called  \emph{standard complex structure} provided 
its determination at the base point $o$ is of the form:
 \begin{equation}\label{complex standard}
  \mathcal{J}_o= 
  \left\{
\begin{aligned}
a \varphi,  & \quad \text{on } \mathfrak{b}_+\\
  \frac{1}{a} \varphi,  & \quad \text{on } \mathfrak{b}_-  
\end{aligned}
\right.,
\end{equation}
where $a$ is a positive constant.

\medskip
A \emph{standard para-complex structure} on $B$ is a $G^o$-invariant almost complex structure on $B$ 
whose determination at the base point $o$ is of the form:
 \begin{equation}\label{paracomplex standard}
  \mathcal{I}_o=
   \left\{
 \begin{aligned}
  a \varphi,  & \quad \text{on } \mathfrak{b}_+\\
  -\frac{1}{a} \varphi,  & \quad \text{on } \mathfrak{b}_-
 \end{aligned}
 \right.,
\end{equation}
where $a$ is a positive constant.
%
\end{definition}

\begin{remark}\em
 A (para-)complex structure $J$ on the vector space $\mathfrak{b}$ defined as in \eqref{complex standard} (resp. \eqref{paracomplex standard}) does not induce in general a $G^o$-invariant almost complex (resp. para-complex) structure on $B$. 
\end{remark}

%


\medskip

\begin{theorem}\label{th_standard}
 Let $(M,\varphi,\xi,\eta,g)$ be a simply connected, complete, contact metric $(\kappa,\mu)$-manifold and let
 $(B,\bar{\nabla})$ the symmetric base space of the canonical fibration of $M$.
Then:
\begin{enumerate}
 \item $|I_M|>1$ if and only if $B$ admits a standard complex structure.
 \item $|I_M|<1$ if and only if $B$ admits a standard para-complex structure.
\end{enumerate}

Moreover, in each case such a standard complex or para-complex structure is uniquely determined; precisely,
it corresponds to the following value of the constant $a$ in \eqref{complex standard}, \eqref{paracomplex standard}:
\begin{equation*}
 a=\sqrt{\frac{I_M+1}{I_M-1}},
\end{equation*} 
when $|I_M|>1$, and
\begin{equation*}
 a=\sqrt{-\frac{I_M+1}{I_M-1}},
\end{equation*}
when $|I_M|<1$. 
\end{theorem}

\begin{proof}
Let $(\mathfrak{b}, [\,,\,,\,])$ the Lie triple system associated to the symmetric space $(B,\bar{\nabla})$. The Lie triple product $[\,,\,,\,]$ is given by the curvature $\bar{R}$ of $\bar{\nabla}$ in the base point $o$:
 \begin{equation*}
  [X,Y,Z]=-\bar{R}_o(X,Y)Z.
 \end{equation*}
Let $J:\mathfrak{b}\rightarrow \mathfrak{b}$ be a complex structure on $\mathfrak{b}$ of the form: 
\begin{equation}\label{J complex}
  J= 
   \left\{
 \begin{aligned}
  a \varphi,  &\quad \text{on } \mathfrak{b}_+ \\
  \frac{1}{a} \varphi,  & \quad\text{on } \mathfrak{b}_-
 \end{aligned}
 \right.,
\end{equation}
where $a$ is a real parameter, $a>0$.

For every $X_+,Y_+,Z_+\in \mathfrak{b}_+$ and $X_-,Y_-,Z_-\in \mathfrak{b}_-$, using \eqref{R symmetric base} and \eqref{J complex}, 
by a direct computation, one can check that:

\begin{equation*}
 \begin{aligned}
   \bar{R}(X_+,Y_+)JZ_+  = & J \bar{R}(X_+,Y_+)Z_+,\quad
   \bar{R}(X_+,Y_+)JZ_-  =& J \bar{R}(X_+,Y_+)Z_-,\\
   \bar{R}(X_-,Y_-)JZ_+  =& J \bar{R}(X_-,Y_-)Z_+,\quad
    \bar{R}(X_-,Y_-)JZ_- =& J \bar{R}(X_-,Y_-)Z_-,\\
     \bar{R}(X_+,Y_-)JZ_- =& \frac{1}{a}(2 \lambda-\mu+2)g(\varphi X_+,Y_-)Z_-,\\		
     J \bar{R}(X_+,Y_-)Z_-=& -a (\mu-2+2\lambda)g(\varphi X_+,Y_-)Z_-.\\
 \end{aligned}
\end{equation*}

Hence, the condition:
$$
\bar{R}(X_+,Y_-)JZ_-=J \bar{R}(X_+,Y_-)Z_-,
$$
is satisfied for every $X_+\in \mathfrak{b}_+$, $Y_-,Z_-\in \mathfrak{b}_-$ if and only if there exists $a>0$ such that $2 \lambda-\mu+2=-a^2 (\mu-2+2\lambda)$.

If $\mu-2+2\lambda=0$ then also $2 \lambda-\mu+2=0$. It follows that $\kappa=1$, but by assumption $M$ is non-Sasakian; then it must be $\mu-2+2\lambda\neq0$ and
\begin{equation*}
 -\frac{2 \lambda-\mu+2}{2\lambda+\mu-2}>0.
\end{equation*}
This condition is equivalent to require that $|I_M|>1$.

\medskip

Finally: 
\begin{equation*}
 \begin{aligned}
   \bar{R}(X_+,Y_-)JZ_+ =&-a(2\lambda+\mu-2)g(\varphi X_+,Y_-)Z_+,\\
   J\bar{R}(X_+,Y_-)Z_+  =& \frac{1}{a}(2\lambda -\mu+2)g(\varphi X_+,Y_-)Z_+
 \end{aligned}
\end{equation*}

%
%
Thus 
$$
\bar{R}(X_+,Y_-)JZ_+  =J\bar{R}(X_+,Y_-)Z_+,
$$
for every $X_+,Z_+\in \mathfrak{b}_+$, $Y_-\in \mathfrak{b}_-$, if and only if there exist $a>0$ such that $2\lambda -\mu+2=-a^2(2\lambda+\mu-2)$.

\medskip

We conclude that the complex structure $J$ is invariant if and only if \\$|I_M|>1$. Moreover in this case 
$$
a= \sqrt{\frac{2-\mu+2 \lambda}{2-\mu-2\lambda}}.
$$

\medskip

With analogous considerations, we obtain that the para-complex structure defined on $\mathfrak{b}$ by:
\begin{equation}\label{I para-complex}
  I= 
   \left\{
 \begin{aligned}
  a \varphi,  &\quad\text{on } \mathfrak{b}_+ \\
  -\frac{1}{a} \varphi,  & \quad \text{on } \mathfrak{b}_-
 \end{aligned}
 \right.,
\end{equation}
where $a>0$, is an invariant para-complex structure if and only if $-1<I_M<1$. In this case 
$$
a= \sqrt{-\frac{2-\mu+2 \lambda}{2-\mu-2\lambda}}.
$$
\end{proof}

\begin{remark}\em
Cappelletti Montano, Carriazo and Mart\'{i}n Molina \cite{Carriazo} showed that every non-Sasakian contact metric $(\kappa,\mu)$-manifold $(M,\varphi,\xi,\eta,g)$ with $|I_M|>1$ admits a Sasakian structure $(\tilde{\varphi},\xi,\eta,\tilde{g})$ obtained by deforming the $(1,1)$-tensor $\varphi$ and the Riemannian metric $g$ as follows:
 \begin{equation*}
 \begin{aligned}
  & \tilde{\varphi}=\epsilon \frac{1}{(1-\kappa)\sqrt{(2-\mu)^2-4(1-\kappa)}}\mathcal{L}_{\xi}h \circ h, \\
  &\tilde{g}=-\di \eta(\cdot , \tilde{\varphi}\cdot)+\eta\otimes \eta,
  \end{aligned}
 \end{equation*}
 where 
 \begin{equation*}
 \epsilon=
  \left\{
 \begin{aligned}
 1 &\quad \text{ if } I_M>1\\
 -1 & \quad \text{ if } I_M<-1
 \end{aligned}
 \right..
 \end{equation*}
 
Moreover, for every point of $M$ there exists a local $CR$-symmetry \cite[Theorem~3.2]{DileoLotta}
Observe that the $CR$-symmetries preserve the tensor field $h$, and hence they preserve also $\tilde{\varphi}$ and $\tilde{g}$. 
By \cite[Proposition~3.3]{DileoLotta} we have that $(M,\tilde{\varphi},\xi,\eta,\tilde{g})$ is a Sasakian $\varphi$-symmetric space
 and then it fibers over a K\"ahler manifold $(B,\bar{\mathcal{J}},\bar{g})$ that is an Hermitian symmetric space \cite{takahashi}. 
One can check that $\bar{\mathcal{J}}$ coincides with the standard complex structure $\mathcal{J}$ on $B$ in our sense.
\end{remark}


\smallskip
\begin{proposition}
The standard (para-)complex structure on the base space $(B,\bar{\nabla})$ of a simply connected, complete, non-Sasakian, contact metric $(\kappa,\mu)$-manifold $M$ with $|I_M|>1$ ( $|I_M|<1$) is actually a twisted (para-)complex $G^o$-invariant structure.
%
\end{proposition}

\begin{proof}
This can be easily verified directly using equation~\eqref{R symmetric base}.
\end{proof}

\medskip
\begin{theorem}\label{base spaces}
 Let $M^{2n+1}$ be a simply connected, complete, non-Sasakian, contact metric $(\kappa,\mu)$-manifold.
  Then:
 \begin{enumerate}
  \item  $I_M>1$ if and only if its twisted complex symmetric base space $(B,\bar{\nabla},\mathcal{J})$ 
  is the complexification $SO(n+2)/(SO(n)\times SO(2))$ of $S^n$. \label{sph}
  \item $-1<I_M<1$ if and only if its twisted para-complex symmetric base space $(B,\bar{\nabla},\mathcal{I})$ is 
the para-complexification $SO(n+1,1)/(SO(n)\times SO(1,1))$ of  $S^n$. \label{para sph}
  \item $I_M<-1$ if and only if its twisted complex symmetric base space $(B,\bar{\nabla},\mathcal{J})$ is the complexification $SO(n,2)/(SO(n)\times SO(2))$ of $\mathbb{H}^n$. \label{hyp}
 \end{enumerate}
\end{theorem}

\begin{proof}
 Consider the Lie triple system $(\mathfrak{b}, [\,,\,,\,])$ associated to the canonical symmetric base space $(B,\bar{\nabla})$.
The Lie triple commutator $[\;,\;,\;]:\mathfrak{b}\times \mathfrak{b}\times\mathfrak{b} \rightarrow \mathfrak{b}$, is given by:
 \begin{equation*}
  [X,Y,Z]=-\bar{R}_o(X,Y)Z,
 \end{equation*}
 where $\bar{R}$ is the curvature of $\bar{\nabla}$. By direct computation, using Proposition~\ref{base space curvature} we see 
 that the linear mapping 
\begin{equation*}
 \tau: X\in \mathfrak{b} \mapsto \frac{1}{\lambda}hX\in \mathfrak{b},
\end{equation*}
is an involutive automorphism of the LTS $(\mathfrak{b}, [\,,\,,\,])$.
 Thus the space $\mathfrak{b}^{\tau}$ of the $\tau$-fixed elements of $\mathfrak{b}$, together with the induced Lie triple bracket, is a Lie triple system. 
 Actually, being 
 $$
 \mathfrak{b}^{\tau}=\mathfrak{b}_+,
 $$ 
 and being the restriction $\bar{R}_+$ of $\bar R$ to $\mathfrak{b}_+$ given by
 $$
 \bar{R}_{+}(X_+,Y_+)Z_+=(2-\mu+2\lambda)(g(Y_+,Z_+)X_+-g(X_+,Z_+)Y_+),
 $$
 we have that the LTS $(\mathfrak{b}_+,-\bar{R}_+)$ is isomorphic to the LTS belonging to the
 sphere $S^n$ or the hyperbolic space $\mathbb{H}^n$, according to the circumstance that the Boeckx invariant $I_M$ is $> -1$ or $<-1$ respectively; indeed
 we have $2-\mu+2\lambda = 2\lambda\Big(I_M+1\Big).$

\bigskip
Suppose $|I_M|>1$. Let $J$ be the twisted complex structure on $\mathfrak{b}$ corresponding to the standard complex structure $\mathcal{J}$ of $B$.
Observe that $J\tau+\tau J=0$, being $\varphi h +h\varphi=0$. 
Then $(\mathfrak{b}, [\,,\,,\,], J,\tau)$ is a twisted complexification of $(\mathfrak{b}_+,-\bar{R}_+)$.

We recall that, by definition, the structure tensor $T$ of $\mathcal{J}$ at the base point $o$ is:
$$
T_o(X,Y)Z=-\frac{1}{2}\Big(\bar{R}_o(X,Y)Z+J\bar{R}_o(X,JY)Z\Big),
$$
and that its restriction $T_+$ to $\mathfrak{b}_+$ yields the Jordan extension
$(\mathfrak{b}_+,T_+)$ of the LTS $(\mathfrak{b}_+,-\bar{R}_+)$,
uniquely associated to its twisted complexification $(\mathfrak{b}, [\,,\,,\,], J,\tau)$
(see Theorem~\ref{Thjor}).

Computing  $T_+$ we obtain:
\begin{equation*}
 \begin{aligned}
  T_{+}(X_+,Y_+)Z_+ &=-\frac{1}{2}\Big( \bar{R}(X_+,Y_+)Z_+ +J\bar{R}(X_+,JY_+)Z_+\Big)\\
                &=\frac{\mu-2-2\lambda}{2}\Big( g(Y_+,Z_+)X_+-g(X_+,Z_+)Y_++g(X_+,Y_+)Z_+\Big).
 \end{aligned}
\end{equation*}
Hence, taking into account the complexification diagrams of the sphere and of the hyperbolic space \cite[Chapter~IV]{Bertram},  
we obtain assertions \ref{sph} and \ref{hyp}. 

\bigskip
Now suppose $|I_M|<1$ and denote by $I$ the twisted para-complex structure on $\mathfrak{b}$ corresponding to the standard para-complex structure $\mathcal{I}$ of $B$ at the base point.
We have that $I\tau+\tau I=0$, being $\varphi h +h\varphi=0$, and hence $(\mathfrak{b}, [\,,\,,\,], I,\tau)$ is a twisted para-complexification of $(\mathfrak{b}^{\tau},-\bar{R}_+)$. The structure tensor of $\mathcal{I}$ at the base point $o$ is:
$$
T_o(X,Y)Z=-\frac{1}{2}\Big(\bar{R}_o(X,Y)Z-I\bar{R}_o(X,IY)Z\Big).
$$
Then the Jordan extension of $\bar{R}_+$ uniquely associated to the twisted para-complexification $(\mathfrak{b}, [\,,\,,\,], I,\tau)$ of the LTS $(\mathfrak{b}_+,-\bar{R}_+)$ is:
\begin{equation*}
 \begin{aligned}
  T(X_+,Y_+)Z_+ &=-\frac{1}{2}\Big( \bar{R}(X_+,Y_+)Z_+ -I\bar{R}(X_+,IY_+)Z_+\Big)\\
                &=-\frac{2-\mu+2\lambda}{2}\Big( g(Y_+,Z_+)X_+-g(X_+,Z_+)Y_++g(X_+,Y_+)Z_+\Big).
 \end{aligned}
\end{equation*}
Then, comparing again with the complexification diagram of the sphere we obtain assertion \ref{para sph}. 

\end{proof}

\section{Homogeneous model spaces of contact metric $(\kappa, \mu)$-spaces}

In this section we complete our classification, showing that one can actually construct a contact metric $(\kappa,\mu)$-space with
prescribed Boeckx invariant starting from each of the symmetric spaces in Table~\ref{tab:tableB}. More precisely, we prove 

\begin{theorem}\label{bbba}
The simply connected,  complete, contact metric $(\kappa,\mu)$-spaces with Boeckx invariant different from $\pm 1$
can be classified as follows.

\medskip
a) The homogeneous space $SO(n,2)/SO(n)$ carries a one parameter family of invariant contact metric $(\kk,\mu)$ structures whose Boeckx invariant assumes all the values in $]-\infty,-1[$.

\medskip
b) The homogeneous space $SO(n+2)/SO(n)$ carries a one parameter family of invariant contact metric $(\kk,\mu)$ structures whose Boeckx invariant assumes all the values in $]1,+\infty[$.

\medskip
c) The homogeneous space $SO(n+1,1)/SO(n)$ carries a one parameter family of invariant contact metric $(\kappa,\mu)$ structures whose Boeckx invariant assumes all the values in $]-1,1[$.
\end{theorem}

%

%
%
%
%
%
%

\begin{proof}
Starting from a fixed Hermitian or para-Hermitian symmetric structure on each of the symmetric spaces
\begin{equation*}
 \begin{aligned}
  & B_1= SO(n+2)/(SO(n)\times SO(2)),\\
  & B_2= SO(n,2)/(SO(n)\times SO(2)),\\
   & B_3= SO(n+1,1)/(SO(n)\times SO(1,1)),\\
 \end{aligned}
\end{equation*} 
we shall construct explicitly a one parameter family of invariant contact metric $(\kk,\mu)$ structures on the homogeneous spaces
\begin{equation*}
 \begin{aligned}
  & M_1=SO(n+2)/SO(n),\\
  & M_2=SO(n,2)/SO(n),\\
   & M_3=SO(n+1,1)/SO(n),
 \end{aligned}
\end{equation*}
with $I_{M_1}>1$, $I_{M_2}<-1$ and $-1<I_{M_3}<1$.

\medskip

We first consider the symmetric Lie algebras $\mathfrak{g}_1:=\mathfrak{so}(n+2)$ and $\mathfrak{g}_2:=\mathfrak{so}(n,2)$ with symmetric decompositions:
\begin{equation*}
\mathfrak{g}_i=\mathfrak{h}_i\oplus \mathfrak{b}_i, 
\end{equation*}
where 
 \begin{equation*}
 \begin{aligned}
\mathfrak{h}_1=\mathfrak{h}_2 &:= \left\{ 
\left[
\begin{array}{c|c}
\begin{matrix}
            0 & -\lambda\\
            \lambda & 0
           \end{matrix} & \mathbf{0}  \\
\hline
\begin{matrix}
           \mathbf{0} & \mathbf{0}
           \end{matrix}   & \mathbf{a}
\end{array}
\right]: \lambda \in \mathbb{R} , \; \mathbf{a}\in \mathfrak{so}(n) \right\}=\mathfrak{so}(2)\oplus\mathfrak{so}(n),\\
 \mathfrak{b}_1 &:= \left\{ 
\left[
\begin{array}{c|c}
\begin{matrix}
            \mathbf{0}
           \end{matrix} & 
           \begin{matrix}
           -\mathbf{v}^T\\
           -\mathbf{w}^T
           \end{matrix}\\
\hline
\begin{matrix}
           \mathbf{v} & \mathbf{w}
           \end{matrix}   & \mathbf{0}
\end{array}
\right]: \mathbf{v}, \mathbf{w} \in \mathbb{R}^n  \right\}\simeq T_oB_1,\\
  \mathfrak{b}_2 &:= \left\{ 
\left[
\begin{array}{c|c}
\begin{matrix}
            \mathbf{0}
           \end{matrix} & 
           \begin{matrix}
           \mathbf{v}^T\\
           \mathbf{w}^T
           \end{matrix}\\
\hline
\begin{matrix}
           \mathbf{v} & \mathbf{w}
           \end{matrix}   & \mathbf{0}
\end{array}
\right]: \mathbf{v}, \mathbf{w} \in \mathbb{R}^n  \right\}\simeq T_oB_2.
\end{aligned}
 \end{equation*}
 The $Ad(SO(2)\times SO(n))$-invariant almost complex structure $J_i:\mathfrak{b}_i\rightarrow\mathfrak{b}_i$ defined by
 \begin{equation*}
  J_i(v \; w)=(-1)^i(w \; -v), 
 \end{equation*}
 and the $Ad(SO(2)\times SO(n))$-invariant metric $G_i$ on $\mathfrak{b}_i$:
 \begin{equation*}
  G_i((v \;w),(u\;z))=<v,u>+<w,z>,
 \end{equation*}
determine an invariant Hermitian symmetric structure $(\mathcal{J}_i,\bar{g}_i)$ on $B_i$; here
 $<,>$ denotes the standard inner product on $\mathbb{R}^n$ and  $(v\; w)$ denotes the matrix
 \begin{equation*}
  \left[
\begin{array}{c|c}
\begin{matrix}
            0 & 0\\
            0 & 0
           \end{matrix} &
            \begin{matrix}
           -\mathbf{w}^T\\
           -\mathbf{v}^T
           \end{matrix}\\
\hline
\begin{matrix}
           \mathbf{v} & \mathbf{w}
           \end{matrix}   & \mathbf{0}
\end{array}
\right]
 \end{equation*}
 in the case $i=1$, resp. the matrix 
 \begin{equation*}
  \left[
\begin{array}{c|c}
\begin{matrix}
            0 & 0\\
            0 & 0
           \end{matrix} &
            \begin{matrix}
           \mathbf{w}^T\\
           \mathbf{v}^T
           \end{matrix}\\
\hline
\begin{matrix}
           \mathbf{v} & \mathbf{w}
           \end{matrix}   & \mathbf{0}
\end{array}
\right]
 \end{equation*}
 in the case $i=2$. 
 Observe that the decomposition of $\mathfrak{g}_i$
  \begin{equation}\label{dec}
 \mathfrak{g}_i=\mathfrak{so}(n)\oplus \mathfrak{m}_i, 
\end{equation}

\begin{equation*}
 \mathfrak{m}_i:=\mathbb{R}\xi\oplus \mathfrak{b}_i, \quad 
 \xi:=\left[
\begin{array}{c|c}
\begin{matrix}
            0 & -1\\
            1 & 0
           \end{matrix} & \mathbf{0}  \\
\hline
\begin{matrix}
           \mathbf{0} & \mathbf{0}
           \end{matrix}   & \mathbf{0}
\end{array}
\right],
\end{equation*}
is a reductive decomposition for $M_i$. Indeed, for every
 \begin{equation*}
 a=\left[
\begin{array}{c|c}
\begin{matrix}
            1 & 0\\
            0 & 1
           \end{matrix} & \mathbf{0}  \\
\hline
\begin{matrix}
           \mathbf{0} & \mathbf{0}
           \end{matrix}   & \mathbf{a}
\end{array}
\right] \in SO(n), \quad X=s\xi + (v\;w) \in \mathfrak{m}_i
\end{equation*}
we have that $Ad(a)X=s\xi + (av\;aw)$. In particular, we have $Ad(a)\xi=\xi$ for every $a\in SO(n)$.

\medskip
We have a natural decomposition of $\mathfrak{b}_i$:
$$\mathfrak{b}_i=\mathfrak{p}_i\oplus\mathfrak{q}_i,$$
where 
$$\mathfrak{p}_i:=\{(v \; 0) \;|\; v\in \mathbb{R}^n\}, \quad \mathfrak{q}_i:=\{(0\;w)\; |\; w\in \mathbb{R}^n\}.
$$ 
By using this decomposition, we define on $\mathfrak{m}_i$ a $(1,1)$ tensor $\varphi_i$, a inner product $g_i$ and a $1$-form $\eta_i$ as follows:
\begin{equation}\label{phi}
\begin{aligned}
 \varphi_i (Z) &:=
  \left\{
 \begin{aligned}
          \alpha J Z  &\quad  \text{ if } Z\in \mathfrak{p}_i \\
          \frac{1}{\alpha} JZ &\quad \text{ if }  Z\in\mathfrak{q}_i\\
          0                  & \quad \text{ if } Z\in \mathbb{R}\xi
         \end{aligned}
         \right.,\\ 
 g_i(X,Y)&:=st+\frac{1}{2} \big(\alpha <v,u>+\frac{1}{\alpha}<w,z>\big),\quad\eta_i(X):=s,
 \end{aligned}
\end{equation}
 where $\alpha>0$,
and $X=s\xi + (v \; w)$, $Y=t\xi + (u \;z)$ are arbitrary elements of $\mathfrak{m}_i$.
These tensors are $Ad(SO(n))$-invariant, indeed for every $a\in SO(n)$:
\begin{equation*}
\begin{aligned}
 Ad(a)\varphi_i X&= Ad(a)\Big((-1)^i\big(\alpha (0 \; -v) +\frac{1}{\alpha} (w \; 0)\big)\Big)\\
                &= (-1)^i \big(\alpha (0 \; -av) +\frac{1}{\alpha} (aw \; 0)\big)\\
                &= \varphi_i Ad(a)X,\\ \\
  g_i(Ad(a)X, Ad(a)Y)&= g\big(s\xi+ (av \; aw), t\xi+ (au \;az)\big)\\
                       &= st+\frac{1}{2}\big(\alpha<av,au>+\frac{1}{\alpha}<aw,az>)\\
                       &= st+\frac{1}{2}\big(\alpha<v,u>+\frac{1}{\alpha}<w,z>)\\
                       &= g(X,Y),
 \end{aligned}
\end{equation*}
finally, being $Ad(a)\xi=\xi$, we also have that $Ad(a)^*\eta_i=\eta_i$. Observe that the invariance of $\eta_i$ implies
that,  for every $X\in \mathfrak{g}_i$ and $Y\in\mathfrak{X}(M_i)$
\begin{equation*}
 0=(\mathcal{L}_{X^*}\eta_i)Y=X^*(\eta_i Y)-\eta_i([X^*,Y])
\end{equation*}
where $X^*$ is the fundamental vector field determined by $X$.
Thus for every $X,Y\in \mathfrak{m}_i$:
\begin{equation*}
\begin{aligned}
 2\di\eta_i(X^*,Y^*) &= X^*(\eta_i Y^*)-Y^*(\eta_i X^*)-\eta_i([X^*,Y^*])\\
                   &= -\eta_i ([Y^*,X^*])\\
                   &= -\eta_i ([X,Y]^*).
\end{aligned}
 \end{equation*}
Evaluating this formula at the base point $o\in M_i$ yields:
\begin{equation}\label{eta}
 2{(\di\eta_i)}_{o}(X,Y)=-\eta_i ([X,Y]_{\mathfrak{m}_i}).
\end{equation}
By direct computations, using \eqref{phi}, \eqref{eta}, we obtain that 
\begin{equation*}
{(\di\eta_i)}_{o}(X,Y)={g_i}(X,\varphi_i Y),\qquad X,Y\in\mathfrak{m}_i.
\end{equation*}
This proves that the invariant tensors $(\varphi_i,\xi,\eta_i,g_i)$ make up a contact metric structure on $M_i$. Moreover it is a $K$-contact structure if and only if $\alpha=1$. Indeed, being $\xi$ and $\varphi_i$ invariant tensors on $M_i$, they are parallel with respect to the canonical connection $\tilde{\nabla}$ associated to the decomposition \eqref{dec}, hence:
\begin{equation*}
\begin{aligned}
 (\mathcal{L}_{\xi}\varphi_i)Y &= [\xi, \varphi_i Y]-\varphi_i [\xi,Y]\\
                             &= \tilde{\nabla}_{\xi}\varphi_i Y-\tilde{T}(\xi,\varphi_i Y)       -\varphi_i\big(\tilde{\nabla}_{\xi} Y-\tilde{T}(\xi, Y)\big)\\
                             &= -\tilde{T}(\xi,\varphi_i Y)+\varphi_i\tilde{T}(\xi, Y);
\end{aligned}
\end{equation*}
then
\begin{equation*}
\begin{aligned}
 2(h_i)_o(v\;w) & = (\mathcal{L}_{\xi}\varphi_i)_o(v\;w)\\
                & =[\xi,\varphi_i(v \; w)]-\varphi_i[\xi,(v\;w)]\\
                & =(-1)^i[\xi, (\frac{1}{\alpha}w \quad -\alpha v)]-\varphi_i (-w \; v)\\
                & =(-1)^i(\alpha v \quad \frac{1}{\alpha}w)-(-1)^i(\frac{1}{\alpha}v \quad \alpha w)\\
                & =(-1)^i\Big(\frac{\alpha^2-1}{\alpha}v \quad -\frac{\alpha^2-1}{\alpha}w\Big).
\end{aligned}                
\end{equation*}
Applying Theorem~\ref{characterization of (k,mu)} we see that $(\varphi_i,\xi,\eta_i,g_i)$ is a contact metric $(\kk,\mu)$ structure on $M_i$ for every $\alpha>0, \alpha\neq1$; moreover, by construction, $\mathcal{J}_i$
is a standard complex structure on the base space $B_i$ of the canonical fibration of $M_i$, in the sense of Definition~\ref{def standard}.
In particular if $0<\alpha<1$ then, by the uniqueness result in Theorem~\ref{th_standard} we must have
\begin{equation*}
\sqrt{\frac{I_{M_1}+1}{I_{M_1}-1}}=\frac{1}{\alpha}, \quad \sqrt{\frac{I_{M_2}+1}{I_{M_2}-1}}=\alpha,
\end{equation*}
or equivalently
\begin{equation*}
 I_{M_1}=\frac{1+\alpha^2}{1-\alpha^2}, \quad I_{M_2}=-\frac{1+\alpha^2}{1-\alpha^2}. 
\end{equation*}
Thus, as $\alpha$ varies in $]0,1[$, $I_{M_1}$ assumes all the values in $]1,+\infty[$ and $I_{M_2}$ assume all the values in $]-\infty, -1[$.

\bigskip

Now we consider the Lie algebra  $\mathfrak{g}:=\mathfrak{so}(n+1,1)$ with symmetric decomposition $\mathfrak{g}=\bar{\mathfrak{h}}\oplus \mathfrak{b}$,
where 
\begin{equation*}
 \begin{aligned}
& \bar{\mathfrak{h}}:= \left\{ 
\left[
\begin{array}{c|c}
\begin{matrix}
            0 & \lambda\\
            \lambda & 0
           \end{matrix} & \mathbf{0}  \\
\hline
\begin{matrix}
           \mathbf{0} & \mathbf{0}
           \end{matrix}   & \mathbf{a}
\end{array}
\right]: \lambda \in \mathbb{R} , \; \mathbf{a}\in \mathfrak{so}(n) \right\}=\mathfrak{so}(1,1)\oplus\mathfrak{so}(n),\\
 & \mathfrak{b}:= \left\{ 
\left[
\begin{array}{c|c}
\begin{matrix}
            \mathbf{0}
           \end{matrix} & 
           \begin{matrix}
           \mathbf{v}^T\\
           -\mathbf{w}^T
           \end{matrix}\\
\hline
\begin{matrix}
           \mathbf{v} & \mathbf{w}
           \end{matrix}   & \mathbf{0}
\end{array}
\right]: \mathbf{v}, \mathbf{w} \in \mathbb{R}^n  \right\}\simeq T_oB_3.
\end{aligned}
 \end{equation*}
Let $(\mathcal{I},\bar{g})$ be the para-Hermitian structure on $B_3$ determined by the $Ad(SO(1,1)\times SO(n))$-invariant structure $(I,G)$ on $\mathfrak{b}$:
\begin{equation*}
\begin{aligned}
 I(v \;w) &:= (w \;v),\\
 G((v \;w), (u \;z))&:= <v,u>-<w,z>,
\end{aligned}
 \end{equation*}
where $(v \;w)$ denotes the matrix 
\begin{equation*}
 \left[
\begin{array}{c|c}
\begin{matrix}
            \mathbf{0}
           \end{matrix} & 
           \begin{matrix}
           \mathbf{v}^T\\
           -\mathbf{w}^T
           \end{matrix}\\
\hline
\begin{matrix}
           \mathbf{v} & \mathbf{w}
           \end{matrix}   & \mathbf{0}
\end{array}
\right]\in \mathfrak{b}.
\end{equation*}
The homogeneous space $SO(n+1,1)/SO(n)$ is reductive with respect to the decomposition
\begin{equation*}
 \mathfrak{so}(n+1,1)=\mathfrak{so}(n)\oplus \mathfrak{m},
\end{equation*}
where 
\begin{equation*}
 \mathfrak{m}:=\mathfrak{so}(1,1)\oplus\mathfrak{b}=\mathbb{R}\xi\oplus \mathfrak{b}, \quad 
 \xi:=\left[
\begin{array}{c|c}
\begin{matrix}
            0 & 1\\
            1 & 0
           \end{matrix} & \mathbf{0}  \\
\hline
\begin{matrix}
           \mathbf{0} & \mathbf{0}
           \end{matrix}   & \mathbf{0}
\end{array}
\right];
\end{equation*}
indeed
\begin{equation*}
 Ad(a)(s\xi+(v\;w))=s\xi+(av\;aw),
\end{equation*}
for every $a\in SO(n)$, $X=s\xi+(v\;w)\in \mathfrak{m}$.

\medskip
Now we consider the natural decomposition of $\mathfrak{b}$:
$$\mathfrak{b}=\mathfrak{p}\oplus\mathfrak{q},$$
where 
$$
\mathfrak{p}:=\{(v \; 0) \;|\; v\in \mathbb{R}^n\} \subset \mathfrak{b},\quad \mathfrak{q}:=\{(0\;w)\; |\; w\in \mathbb{R}^n\} \subset \mathfrak{b}.
$$
Using this decomposition, we define on $\mathfrak{m}$ the following $Ad(SO(n))$-invariant tensors:
\begin{equation}
\begin{aligned}
 \varphi (Z) &:=
  \left\{
 \begin{aligned}
          \alpha JZ   & \quad \text{ if } Z\in\mathfrak{p} \\
          -\frac{1}{\alpha} JZ &\quad \text{ if }  Z\in\mathfrak{q}\\
          0                  & \quad \text{ if } Z\in\mathbb{R}\xi
         \end{aligned}
         \right.,\\ 
 g(X,Y)&:=st+\frac{1}{2} \big(\alpha <v,u>+\frac{1}{\alpha}<w,z>\big),\quad\eta(X):=s,
 \end{aligned}
\end{equation}
where  $\alpha>0$ and $X=s\xi + (v \; w)$, $Y=t\xi + (u \;z)$ are any matrices in $\mathfrak{m}$.
One checks by the same method used above that $(\varphi,\xi,\eta,g)$ is a contact metric $(\kk,\mu)$--structure. 
Moreover 
\begin{equation*}
 2h_o(v\;w)=( -\frac{\alpha^2+1}{\alpha}v \quad \frac{\alpha^2+1}{\alpha}w).
\end{equation*}
Then applying again Theorem~\ref{th_standard} we get $$I_{M_3}=\frac{\alpha^2-1}{\alpha^2+1},$$ 
and hence, as $\alpha$ varies in $\mathbb{R}_+^*$, $I_{M_3}$ assumes all the values in $]-1,1[$.
\end{proof}

\begin{remark}\em
Of course, in the case $I>1$ we recover, up to isomorphism, the unit tangent sphere bundle $T_1M$ of a Riemannian manifold $(M,g)$ with constant sectional curvature $c>0$, $c\neq 1$.

In the case $I<-1$, we obtain a new homogeneous representation of the contact metric $(\kappa,\mu)$
manfolds $M$ with $I_M< -1$, different from the Lie group representation furnished by
Boeckx.  Actually these models can be geometrically interpreted also as tangent
hyperquadric bundle over Lorentzian space forms, as showed in \cite{Loi-Lo}.
\end{remark}

\begin{remark}\em
The homogeneous model spaces of the contact metric $(\kappa,\mu)$-manifolds here obtained also appear in the classification list of the simply connected sub-Riemannian symmetric spaces carried out by Bieliavsky, Falbel and Gorodski in \cite{FalbelGor}. However, in their paper the contact metric structures are not considered. 
\end{remark}

\vspace{1 cm}

\noindent 
Eugenia Loiudice

\vspace{1 mm}
\noindent Dipartimento di Matematica, Universit\`a di Bari Aldo Moro, 

\noindent Via Orabona 4, 70125 Bari, Italy

\noindent \emph{e-mail}: eugenia.loiudice@uniba.it

\vspace{3 mm}
\noindent 
Antonio Lotta

\vspace{1 mm}
\noindent Dipartimento di Matematica, Universit\`a di Bari Aldo Moro, 

\noindent Via Orabona 4, 70125 Bari, Italy

\noindent \emph{e-mail}: antonio.lotta@uniba.it

\end{document}